\documentclass[10pt,amscd]{amsart}
\usepackage{latexsym}
\usepackage{amsmath}
\usepackage{amssymb}
\usepackage[all]{xy}
\newcommand{\LL}{\mathbb{L}}
\newcommand{\Z}{\mathbb{Z}}
\newcommand{\Q}{\mathbb{Q}}

\newcommand{\C}{\mathbb{C}}

\newcommand{\rank}{{\rm rank}}

\newtheorem{thm}{Theorem }[section]
\newtheorem{prop}[thm]{Proposition}
\newtheorem{lem}[thm]{Lemma}
\newtheorem{claim}[thm]{Claim}
\newtheorem{defi}[thm]{Definition}
\newtheorem{cor}[thm]{Corollary}
\newtheorem{exmp}[thm]{Example}
\newtheorem{rem}[thm]{Remark}
\newtheorem{prob}[thm]{Problem}
\newtheorem{ques}[thm]{Question}

\begin{document}
\title{An induced   map between rationalized classifying spaces for fibrations }
\author{ Toshihiro Yamaguchi }
\address{Kochi University, 2-5-1, Kochi, 780-8520, JAPAN}
\email{tyamag@kochi-u.ac.jp}
\maketitle

\begin{abstract}
Let  $B{ aut}_1X$ be  the Dold-Lashof classifying space of orientable fibrations with fiber $X$ \cite{DL}.
For  a rationally  weakly trivial
map $f:X\to Y$,
our {\it strictly induced } map  $a_f: (Baut_1X)_0\to (Baut_1Y)_0$  
induces  a natural map from a $X_0$-fibration  to a  $Y_0$-fibration. 
It is given by a map 
between  the  differential graded Lie algebras of derivations of  Sullivan models \cite{S}. 
We note  some conditions that the map $a_f$  admits  a section and note some relations with the Halperin conjecture \cite{Hal}.
Furthermore we give the obstruction class for a lifting of a classifying map
$h: B\to (Baut_1Y)_0$   and apply it for liftings of  $G$-actions on $Y$ for a compact connected  Lie group $G$ \cite{Gol} as the case of  $B=BG$ and 
evaluating of rational toral ranks \cite{H}
as  $r_0(Y)\leq r_0(X)$. 
\end{abstract}


\section{Introduction}

Let $X$ (and also $Y$) be a connected and simply connected  CW complex  with $\dim \pi_*(X)_{\Q} <\infty$ 
($G_{\Q}=G\otimes \Q$)
and  $B{ aut}_1X$ be  the Dold-Lashof classifying space of orientable fibrations \cite{DL}.
Here $aut_1X=map(X,X;id_X)$ is the identity component of the space $autX$ of self-equivalences of $X$.
Then any orientable fibration $\xi$ with fibre $X$ over a base space $B$ is  the  pull-back of a universal fibration 
$X\to E_{\infty}^X\to B{ aut}_1X$
by a map $h:B\to   B{ aut}_1X$
and equivalence classes of $\xi$ are classified by their homotopy classes 
\cite{DL}, \cite{Sta}, \cite{Al} (So the map $h$ is often said as the classifying map for the  fibration $\xi$).  
The {Sullivan minimal model} $M(X)$ (\cite{S}) 
determines the rational homotopy type  of $X$,
the homotopy type of the rationalization $X_0$ \cite{HMR} of $X$.
The differential graded Lie algebra (DGL) $DerM(X)  $, the negative derivations of $M(X)$
(see \S 2),  gives rise to a Quillen model for $Baut_1X$ due to Sullivan \cite{S} (cf.\cite{T}, \cite{G}), 
i.e., the spatial realization $||DerM(X)||$ is $(Baut_1X)_0$.
Therefore  we obtain
a map
$ (Baut_1X)_0\to (Baut_1Y)_0$
if there is a DGL-map
$Der  M(X)\to DerM(Y)  $.
However it does {\it not} exist in general.

Let  $f:X\to Y$ be a map whose homotopy fibration $\xi_f: F_f\to X\to Y$ is given by the relative  model
(Koszul-Sullivan extension)
$$M(Y)=(\Lambda V,d)\to (\Lambda V\otimes \Lambda W,D)\simeq M(X)$$
for a certain differential $D$ with $D\mid_{\Lambda V}=d$, where $M(F_f)\cong  ( \Lambda W,\overline{D})$
for the homotopy fiber $F_f$ of $f$ \cite{FHT}.
In this paper, we say a map is {\it rationally  weakly trivial} (abbr., $\Q$-w.t.)
if  $\xi_f$ is rationally  weakly trivial; i.e.,
$\pi_*(X)_{\Q}=\pi_*(F_f)_{\Q}\oplus \pi_*(Y)_{\Q}$.
Then  $(\Lambda V\otimes \Lambda W,D)$
is just the minimal model $M(X)$ of $X$.

\begin{defi}
We say that a $\Q$-w.t. map $f:X\to Y$ strictly induces the map  $$a_f: (Baut_1X)_0\to (Baut_1Y)_0$$
if its Quillen model is given by   the DGL-map $$b_f:Der  (\Lambda V\otimes \Lambda W,D)\to Der(\Lambda V,d)  $$ 
 given by $b_f (\sigma )={\rm proj}_V\circ \sigma  $ with $||b_f||= a_f$.
Here ${\rm proj}_V:\Lambda V\otimes \Lambda W\to \Lambda V$  is the algebra map with 
${\rm proj}_V(w)=0$ for $w\in W$ and  ${\rm proj}_V\mid_{\Lambda V}=id_{\Lambda V}$.
\end{defi}

Two fibration $\xi_{f_1}$ and $\xi_{f_2}$ are fibre homotopy  equivalent 
if there is a  diagram:
$${\small  
\xymatrix{ F_{f_1}\ar[r]^{i_1}&
 X_1\ar[r]^{f_1}& Y\ar@{=}[d]\\
F_{f_2}\ar[u]^{\sim}_{\overline{\psi}}\ar[r]^{i_2}& 
{ X_2}\ar[u]_{{\exists \psi }}^{\sim}\ar[r]^{f_2} & Y ,}
}$$
where $\psi\circ i_2\simeq i_1\circ  \overline{\psi}$ and $f_1\circ \psi=f_2$.
Then  its Sullivan model is given as 
$${\small 
\xymatrix{ \Lambda V ,d\ar[r]\ar@{=}[d]&\Lambda V\otimes \Lambda W,D_1\ar[r]\ar[d]_{\psi}^{\cong}& \Lambda W,
\overline{D}_1\ar[d]^{\cong}_{\overline{\psi}}\\
\Lambda V,d \ar[r]&\Lambda V\otimes \Lambda W,D_2\ar[r]& \Lambda W,\overline{D}_2,
}
}$$
where the left square is DGA-commutative and the right square is DGA-homotopy commutative.
\begin{lem}\label{equi}
Suppose that  two maps $f_1$ and $f_2$ strictly induce $a_{f_1}$ and $a_{f_2}$, respectively. 
 If $\xi_{f_1}$ and $\xi_{f_2}$ are  fibre homotopy equivalent,
there is a DGL-isomorphism $\phi :Der  (\Lambda V\otimes \Lambda W,D_1)\cong Der  (\Lambda V\otimes \Lambda W,D_2)$
such that  $\phi (\sigma )=\psi \circ \sigma \circ \psi^{-1}$ and    $ b_{f_2}\circ \phi =b_{f_1}$.
Thus there is a homotopy equivalence map $\phi : (Baut_1X_1)_0\overset{\sim}{\to} (Baut_1X_2)_0$
such that $a_{f_2}\circ \phi=a_{f_1}$, i.e.,  $ a_{f_1}$ and $a_{f_2}$ are  fibre homotopy  equivalent as fibrations. 
\end{lem}

Let $\min  \pi_*(S)_{\Q}:=\min \{ i>0\mid \pi_i(S)_{\Q}\neq 0\}$ and  $\max  \pi_*(S)_{\Q}:=\max \{ i\geq 0\mid \pi_i(S)_{\Q}\neq 0\}$
 for a space $S$.
In particular, $\min  \pi_*(S)_{\Q}:=\infty$ when $S$ is the  one point space.

\begin{defi}\label{sepa}
A fibration $\xi_f :F_f\to X\overset{f}{\to} Y$ or a map $f:X\to Y$ with homotopy fiber $F_f$ is said to be $\pi_{\Q}$-{\it separable} if
 $\min \pi_*(F_f)_{\Q}\geq \max \pi_*(Y)_{\Q}.$
\end{defi}

If  a map $f:X\to Y$ is  $\pi_{\Q}$-{\it separable}, it is   $\Q$-w.t.
The  condition to be $\pi_{\Q}$-{separable} is equivalent to the  condition that 
$\min  W=\min \{ i>0\mid W^i\neq 0\}\geq \max  V=\max \{ i>0\mid V^i\neq 0\}$ in the relative minimal model $M(Y)=(\Lambda V,d)\to (\Lambda V\otimes \Lambda W,D)$ of $\xi_f$.

\begin{rem}  Recall   a question related to Gottlieb \cite[\S 5]{G2}:
  {\it Which  map $f:X\to Y$ can be extended to a map 
  between  fibrations over a fixed base space $B$,
  that is,
for any fibration $\xi :X\to E\to B$,
does there exist
a fibration $\eta :Y\to E'\to B$ and 
 a map $f':E\to E'$ in  the diagram:
  $${\small  
\xymatrix{ X\ar[d]_{f}\ar[r]^{i}&
 E\ar@{.>}[d]_{{ f'}}\ar[r]^{p}& B\ar@{=}[d]\\
Y\ar[r]^{i'}& 
{ E'}\ar[r]^{p'} & B
\\
}}$$
where $f'\circ i\simeq i'\circ f$ and $p=p'\circ f'$ ?} \cite[Example 3.8]{Y}.
If a map $f:X\to Y$ is  $\pi_{\Q}$-separable, then due to Sullivan minimal model theory, it is obvious that
there are a fibration $\eta$ after rationalization and a map $f':E_0\to E'$ in the diagram: $${\small  
\xymatrix{ X_0\ar[d]_{f_0}\ar[r]^{i_0}&
 E_0\ar@{.>}[d]_{{ f'}}\ar[r]^{p_0}& B_0\ar@{=}[d]\\
Y_0\ar[r]^{i'}& 
{ E'}\ar[r]^{p'} & B_0
\\
}}$$ 
where $f'\circ i_0\simeq i'\circ f_0$ and $p_0=p'\circ f'$.
In particular,  let $f:X\to X(n)$ be the rationalized  Postnikov $n$-stage map of $X$, where 
$\pi_{>n}(X(n))=0$.
Then there exists the  rationally fibre-trivial fibration $\eta$ such that
 $${\small  
\xymatrix{ X\ar[d]_{f}\ar[r]^{i}&
 E\ar@{.>}[d]_{{ f'}}\ar[r]^{p}& B\ar[d]^{l_0}\\
X(n)\ar[r]^{}& 
{ X(n)\times B_0}\ar[r] & B_0
\\
}}$$
homotopically commutes when the rationalized classifying map 
$B\to (Baut_1X)_0\overset{a_f}\to ( Baut_1X(n))_0$
is homotopic to the constant map for  a sufficiently small  $n$.
\end{rem}

\begin{prop}\label{sep}
A $\Q$-w.t. map $f:X\to Y$ strictly induces  $a_f: (Baut_1X)_0\to (Baut_1Y)_0$  if and only if $f$ is  $\pi_{\Q}$-{separable}.
\end{prop}

\begin{thm}\label{xy}For 
a $\pi_{\Q}$-separable map $f:X\to Y$, let 
 $${\small  
\xymatrix{ 
B_0\ar[d]_{{ g}}\ar[r]^{h\ \ }& (Baut_1X)_0\ar[d]^{a_f}\\
{ B'_0}\ar[r]^{h'\ \ } & (Baut_1Y)_0
\\
}}$$
be a  commutative diagram.
Then there exists a map between  total spaces $k :E\to E'$ in the diagram:
$${\small  
\xymatrix{ X_0\ar[d]_{f_0}\ar[r]^{i}&
 E\ar@{.>}[d]_{{ k}}\ar[r]^{p}& B_0\ar[d]^{g}\\
Y_0\ar[r]^{i'}& 
{ E'}\ar[r]^{p'} & B'_0
\\
}}$$
where $k\circ i\simeq i'\circ f_0$ and $g\circ p=p'\circ k$.
Here $p:E\to B_0$ and $p':E'\to B'_0$ are induced by the rationalized classifying maps $h$ and $h'$, respectively. 
\end{thm}

Let $f:X\to Y$ be a map with a  section $s$, i.e., there is a map $s:Y\to X$ with $f\circ s\simeq id_Y$.
Then there is a map $\psi_f :aut_1X\to aut_1Y$
with $\psi_f(g):=f\circ g\circ s$ for $g\in aut_1X$. 
In general, this does not preserve the monoid structures.

\begin{thm}\label{section}
If  a $\pi_{\Q}$-separable map $f$ admits a  section,
  $\Omega a_f\simeq (\psi_f)_0$. 
\end{thm}

In \S 2, we give the proofs under some preparations of models of \cite{FHT} and \cite{T}. 
 In this paper, we consider only  $\Q$-w.t. maps.
For example, we do not consider the inclusion map $i_X:X\to X\times Y$, which is not  $\Q$-w.t.
However $i_X$ induces the monoid map $\psi :aut_1X\to aut_1(X\times Y) $ by $\psi (g)=g\times 1_Y$
and therefore there exists the induced map $B\psi :Baut_1X\to B aut_1(X\times Y) $ without rationalization.
The DGL  model  is  given by the natural inclusion 
$Der M(X)\to Der (M(X)\otimes M(Y))$, which is a DGL-map.

Let ${\rm Sep}_{\Q}$ be the  category that the objects are  simply connected CW-complexes 
of finite dimensional rational homotopy groups
and morphisms are  $\Q$-separable maps. 
When $f:X\to Y$ and $g:Y\to Z$ are  $\pi_{\Q}$-separable maps,
$g\circ f:X\to Z$ is also a  $\pi_{\Q}$-separable map.
Then $$a_{g\circ f}=a_g\circ a_f: (Baut_1X)_0\to (Baut_1Z)_0$$
by our construction. 
In particular $f={\rm id}_X:X\overset{=}{\to} X$
is  $\pi_{\Q}$-separable
and then  $a_f$ is of course the identity map of $(Baut_1X )_0$.
Namely, 
$(Baut_1)_0$ is a functor from ${\rm Sep}_{\Q}$ to $ho(\Q$-$CW_1)$
by $(Baut_1)_0(f):=a_f$.
Here $ho(\Q$-$CW_1)$ is the homotopy category of rational simply connected CW-complexes 
of finite dimensional rational homotopy groups.
This functor is not essentially surjective on objects,
i.e., there are rational spaces that cannot be  realized as $(Baut_1X )_0$  for any $X$ due to Lupton-Smith \cite{LS}(\cite{Sm}).
 Also we can easily find  in Example \ref{real}  there exists  a map  $ (Baut_1X)_0\to (Baut_1Y)_0$ that cannot be  realized as $a_f$  for any  $\pi_{\Q}$-separable map $f$.\\


In \S 3, 
we give some  such conditions for
\begin{ques}\label{que2}
When does (is) the strictly induced map   $a_f: (Baut_1X)_0\to (Baut_1Y)_0$ admit a   section  (fibre-trivial as a fibration) ?
\end{ques}
Some results for this question  are obtained by   Proposition \ref{Tanre} induced from \cite[VI .1.(3) Proposition]{T} that 
the DGL-model of the homotopy fibration $\chi_f: \ F_{a_f}\overset{}{\to} (Baut_1X)_0\overset{a_f}{\to} (Baut_1Y)_0$ is given by
$$ Der  ( \Lambda W, \Lambda V\otimes \Lambda W)\to Der  (\Lambda V\otimes \Lambda W)\overset{b_f}\to Der  (\Lambda V).$$
Let $aut_1f$ be the identity component of the space of all fibre-homotopy self-equivalences of $f$, i.e., 
 $\{ g:X\to X\mid f\circ g=f \ \}$
and  $Baut_1f$  be the classifying space of this topological monoid.
It is known that $Baut_1f\simeq \widetilde{map}(Y,Baut_1(F_f);h)$,
where $h:Y\to Baut_1(F_f)$ is the classifying map of the fibration $F_f\to X\overset{f}{\to} Y$ and
$  \widetilde{map}$ denotes the universal cover of the function space \cite{BHMP}.
 Notice that 
$$Der   ( \Lambda W, \Lambda V\otimes \Lambda W)=Der_{\Lambda V}  ( \Lambda V\otimes \Lambda W)$$
where 
$Der_{\Lambda V}  ( \Lambda V\otimes \Lambda W)$
is the sub DGL of $Der  ( \Lambda V\otimes \Lambda W)$
sending the elements of $\Lambda V$ to zero 
and it  is a Quillen model of $Baut_1f$ 
when $Y$ and $F_f$ are finite
\cite[Theorem 1]{BS} (\cite{FLS}).
Thus we have
\begin{prop}\label{fibre}
When $Y$ and the homotopy fiber $F_f$ are finite
for $f:X\to Y$, the homotopy fiber $F_{a_f}$ of $a_f$ has the rational   homotopy type of  $Baut_1f$.
\end{prop}

It is suitable since $F^{-1}(id_Y)=aut_1f$
if there exists a map $F: aut_1X\to aut_1Y$ such that $f\circ g=F(g)\circ f$ for $g\in aut_1X$.

A space $X$ is said to be {\it elliptic} if 
the dimensions of the rational cohomology algebra and homotopy group are both  finite \cite{FHT}.
An elliptic space $X$ is said to be {\it pure}
if $dM(X)^{even}=0$ and $dM(X)^{odd}\subset M(X)^{even}$.
Furthermore a pure space is said  to be an {\it $F_0$-space} (or {\it positively elliptic}) if 
$\dim \pi_{\rm even}(X)\otimes \Q=\dim \pi_{\rm odd}(X)\otimes \Q$ and $H^{\rm odd}(X;\Q )=0$.
Then it is equivalent to 
$ H^*(X;\Q )\cong  \Q [x_1,\cdots ,x_n]/(f_1,\cdots ,f_n),$
in which $|x_i|$, the degree of $x_i$,
 are even and $f_1,\cdots ,f_n$ forms  a regular sequence in the   $\Q$-polynomial algebra $ \Q [x_1,\cdots ,x_n]$,
where
$M(X)=( \Q [x_1,\cdots ,x_n]\otimes \Lambda (y_1, \cdots ,y_n),d)$ with $dx_i=0$ and $dy_i=f_i$.
In 1976,
S. Halperin \cite{Hal} conjectured that the Serre spectral sequences of all fibrations $X\to E\to B$ of simply connected CW complexes collapse at the $E_2$-terms 
for any $F_0$-space $X$ \cite{FHT}.
For compact connected Lie groups $G$ and $H$ where 
$H$ is a subgroup of $G$,
when ${\rm rank}\ G={\rm rank}\ H$,
the homogeneous space $G/H$ satisfies the Halperin conjecture \cite{ST}.
Also 
the Halperin conjecture is true when $n\leq 3$ \cite{Lu}.
In this paper we note  some  relations with the  Halperin  conjecture \cite[\S 39]{FHT}
 due to W. Meier \cite{M} as 

\begin{thm}\label{two}
Let $Y$ be an $F_0$-space.
Then the fibration $\chi_f$ is fibre-trivial
for any $\pi_{\Q}$-separable map 
$f:X\to Y$
if and only if 
$Y$
 satisfies the  Halperin's conjecture.
\end{thm}

In \S 4, we observe  the cellular obstruction for the lifting 
%
  ${\tilde{h}}$ for a map $h:B\to  (Baut_1Y)_0$:
$${\small  \xymatrix{&  (Baut_1X)_0\ar[d]^{a_f}\\
B \  \ar[r]_{h\ \ \ \ \ \  }\ar@{.>}[ur]^{\tilde{h}}&\ (Baut_1Y)_0
}}$$
for a  $\pi_{\Q}$-separable map $f:X\to Y$.
Of course, it is sufficient to define as  $ \tilde{h}=s\circ h$ if $a_f$ admits a section $s$.
Thus it is a general approach for Question \ref{que2}, which is the case of $B=Baut_1Y$.
Specifically, for a $\pi_{\Q}$-separable map   $f:X\to Y$, let 
$${\small  \xymatrix{B\ar[r]^{h_X\ \ \ }\ar[d]_i&  (Baut_1X)_0\ar[d]^{a_f}\\
B\cup_{\alpha} e^N \  \ar[r]_{h_Y}&\ (Baut_1Y)_0
}}$$
be a  commutative diagram.
Then, from Proposition \ref{fibre},  we define an obstruction class by  derivations  in Theorem \ref{obs} so that
\begin{thm}\label{obstr} Let $f:X\to Y$ 
be a $\pi_{\Q}$-separable map with $Y$ and $F_f$  finite.
There is a lift $h$ such that
$${\small  \xymatrix{B\ar[r]^{h_X\ \ \ }\ar[d]_i&  (Baut_1X)_0\ar[d]^{a_f}\\
B\cup_{\alpha} e^N \  \ar[r]_{h_Y}\ar@{.>}[ur]^{{h}}&\ (Baut_1Y)_0
}}$$
is  commutative if and only if 
${\mathcal O}_{\alpha}(h_X, h_Y)=0$
in $\pi_{N-1}(Baut_1f)_{\Q}$.
 \end{thm}
From Theorem \ref{xy} and Theorem \ref{obstr}, we  obtain

\begin{cor}\label{X-Y} Let $f:X\to Y$  
be a $\pi_{\Q}$-separable map with $Y$ and $F_f$  finite. 
Suppose that there is a fibration $Y\to E\to B$.
If  $H^{n+1}(B,\pi_{n}(Baut_1f)_{\Q})=Hom (H_{n+1}(B), \pi_n(Baut_1f)_{\Q})=   0$ for all $n$, 
there exist  a fibration $X_0\to \tilde{E}\to B_0$ 
and 
 a map 
between the fibrations:
  $${\small  
\xymatrix{ X_0\ar[d]_{f_0}\ar[r]^{}&
 \tilde{E}\ar@{.>}[d]_{{ \tilde{f}}}\ar[r]^{}& B_0\ar@{=}[d]\\
Y_0\ar[r]^{}& 
{ E}\ar[r] & B_0.
\\
}}$$
\end{cor}

\begin{rem}
Let $g:Y\to Z$ and $f:X\to Y$ be $\pi_{\Q}$-separable maps given by the models
$M(Z)=\Lambda U\to \Lambda (U\oplus V)=M(Y)$
and $M(Y)=\Lambda (U\oplus V)\to \Lambda (U\oplus V\oplus W)=M(X)$, respectively.
Then $\pi_*(Baut_1g\circ f)_{\Q}=
H_{*-1}(Der_{\Lambda U}(\Lambda (U\oplus V\oplus W))=
H_{*-1}(Der_{\Lambda U}(\Lambda (U\oplus V))\oplus H_{*-1}(Der_{\Lambda (U\oplus V)}(\Lambda (U\oplus V\oplus W))=
\pi_*(Baut_1 g)_{\Q}\oplus \pi_*(Baut_1f)_{\Q}$.
\end{rem}

\vspace{0.5cm}

In \S 5, we consider an application to lifting actions.
Let $G$ be a topological  group and acts on a CW complex $Y$.
Recall the problem of lifting (up to homotopy) of D. H. Gottlieb \cite{Gol}:

\begin{prob}
When is  a fibration $F_f\to X\overset{f}{\to} Y$ fibre homotopy equivalent to a $G$-fibration ?
i.e.,  when is there  a fibartion $f':X'\to Y$ such that 
$f'$ is fibre homotopy equivalent to $f$ and there is a $G$-action on $X'$ such that $f'$ is equivariant ? 
\end{prob}


Suppose that $G$ is a  compact connected Lie group.
Since $H^*(BG;\Q )$ is evenly graded, 
the  obstruction classes of  Theorem \ref{obstr} are  contained in $\pi_{odd}(Baut_1f)_{\Q}$ 
when $B=BG$.
If $\pi_{odd}(Baut_1f)_{\Q}=0$, they vanish and there exists a lift $h:BG\to  (Baut_1X)_0$.
Then, from Theorem \ref{xy} in the case that $B=B'=BG$ and $g=(id_{BG})_0$,
we obtain by using Theorem  \ref{Gott} of  D. H. Gottlieb

\begin{thm}\label{lift} Let $f:X\to Y$  
be a $\pi_{\Q}$-separable map 
with  $Y$ and $F_f$  finite.
Suppose that a compact Lie group $G$ acts on $Y$.
If  $\pi_{odd}(Baut_1f)_{\Q}=0$, 
the action on $Y$ is rationally lifted to $X$, i.e., 
$f$ is rationally fibre homotopy equivalent to 
a $G$-equivariant map $f':X'\to Y$ for a $G$-space $X'$. 
\end{thm}

Due to Theorem \ref{two} and the result of H. Shiga - M. Tezuka \cite{ST},
we have 
\begin{cor}
Let $f:X\to Y$  
be a $\pi_{\Q}$-separable map such that
  $Y$ is a homogeneous space $G/H$ with $\rank\ G=\rank\ H$.
Then any group action on  $Y$ is rationally lifted to $X$. 
In particular,  the natural $G$-action on $Y$ is rationally lifted to $X$. 
\end{cor}

Furthermore  we apply the obstruction argument to a rational homotopical invariant: 
Let $ r_0(X)$ be the {\it rational  toral rank}
 of a simply connected  complex  $X$
of $\dim H^*(X;\Q )<\infty $,
 i.e., the largest integer $r$ such that an $r$-torus
 $T^r=S^1 \times\dots\times S^1$($r$-factors)  can act continuously
 on a CW-complex $X'$ in  the rational homotopy type of $X$
 with all its isotropy subgroups finite (almost free action)   \cite{AP}, \cite{FOT}, \cite{H}.
It is very difficult to calculate  
$r_0(\ \ )$ in general.
From the definition, we have the inequality $r_0(X\times Y)\geq r_0(X)+r_0(Y)$.
Notice that it may  sometimes be a strict inequality
since there is an example that $r_0(X\times S^{12})>0$
even though $r_0(X)=r_0(S^{12})=0$ \cite[Example 3.3]{JL}.
For a map $f:X\to Y$, we see  $r_0(Y)\leq r_0(X)$ when 
$X= F\times Y$ for any space $F$ and $f$ is the projection $F\times Y\to Y$.
In general, 
 when does a map $f:X\to Y$ induce  $r_0(Y)\leq r_0(X)$ ?

\begin{cor}\label{coro} Let $f:X\to Y$  
be a $\pi_{\Q}$-separable map 
with  $Y$ and $F_f$  finite.
If  $\pi_{odd}(Baut_1f)_{\Q}=0$, we have
$r_0(Y)\leq r_0(X)$. 
\end{cor}



\section{Sullivan models, derivations and Quillen models}
Let
$M(X)=(\Lambda {V},d)$
be the  Sullivan minimal model of simply connected CW complex $X$ of finite type \cite{S}.
  It is a free $\Q$-commutative {\bf d}ifferential {\bf g}raded {\bf a}lgebra (DGA)
 with a $\Q$-graded vector space $V=\bigoplus_{i\geq 1}V^i$
 where $\dim V^i<\infty$ and a decomposable differential,
 i.e., $d(V^i) \subset (\Lambda^+{V} \cdot \Lambda^+{V})^{i+1}$ and $d \circ d=0$.
 Here  $\Lambda^+{V}$ is
 the ideal of $\Lambda{V}$ generated by elements of positive degree.
The degree of a homogeneous element $x$ of a graded algebra is denoted as $|{x}|$.
Then  $xy=(-1)^{|{x}||{y}|}yx$ and $d(xy)=d(x)y+(-1)^{|{x}|}xd(y)$.
Note that  $M(X)$ determines the rational homotopy type of $X$, namely 
the spatial realization is given as $||M(X)||\simeq X_0$.
In particular,  $$V^n\cong {\rm Hom}(\pi_n(X),\Q) \mbox{\ \ and\ \ }H^*(\Lambda {V},d)\cong H^*(X;\Q ).$$
Here the second is an isomorphism as graded algebras. 
Refer to \cite{FHT} for detail.


Let $Der_i M(X)$ be the set of $\Q$-derivations of $M(X)$
decreasing the degree by $i$
with
$\sigma (xy)=\sigma (x)y+(-1)^{i\cdot |x|}x\sigma (y)$
for $x,y\in M(X)$. 
The boundary operator $\partial : Der_i M(X)\to Der_{i-1} M(X)$
is defined by $$\partial  (\sigma)=d\circ \sigma-(-1)^i\sigma \circ d$$
for $\sigma\in Der_iM(X)$.  
We denote  $\oplus_{i>0} Der_iM(X)$ by
$DerM(X)$
in which $Der_1M(X)$ is $\partial$-cycles.
Then ${Der}M(X)$ is a (non-free) DGL by the Lie bracket $$[\sigma ,\tau]:=\sigma\circ \tau-(-1)^{|\sigma||\tau|}
\tau\circ\sigma.$$
Note that $H_*(DerM)=H_*(Der N)$ when free  DGAs $M$ and $N$ are quasi-isomorphic \cite{Sa}.
 Furthermore, recall the definition  of D.Tanr\'{e} \cite[p.25]{T}: 
Let  $(L,\partial )$ be a DGL of finite type with positive degree. Then 
$C^*(L,\partial )=(\Lambda s^{-1}\sharp L, D=d_1+d_2) $ with
$$ \langle d_1s^{-1}z; sx\rangle =-\langle z;\partial x\rangle \mbox{ \ and\ } 
 \langle d_2s^{-1}z; sx_1,sx_2\rangle =(-1)^{|x_1|}\langle z;[x_1,x_2]\rangle,$$
where $\langle s^{-1}z;sx\rangle=(-1)^{|z|}\langle z;x\rangle$ 
and $\sharp L$ is the dual space of $L$.

\begin{thm}\cite[\S 11]{S},\cite{T},\cite{G} 
For a Sullivan model $M(X)$ of $X$,
$Der M(X) $ is a DGL-model of $Baut_1X$.
In particular, there is an isomorphism of graded Lie algebras
$H_*(Der M(X))\cong \pi_*(\Omega Baut_1X)_{\Q}$
where the right-hand has the Samelson bracket. 
Furthermore 
$C^*(Der M(X) )$ is a DGA-model of $Baut_1X$. 
\end{thm}

Let 
$L(X)= (\LL U,\partial )$ be the Quillen model of $X$   \cite[III.3.]{T}, \cite[\S 24]{FHT}.
  It is a free $\Q$-commutative {\bf d}ifferential {\bf g}raded {\bf L}ie algebra (DGL)
 with a $\Q$-graded vector space $U=\bigoplus_{i\geq 1}U_i$
 where $\dim U_i<\infty$ and  $\partial (U_i) \subset (\LL U)_{i-1}$, which is the space of elements of $\LL U$
with  degree $i-1$.
Note that $[x,y]=-(-1)^{|x||y|}[y,x]$ and Jacobi identity:
$$[x,[y,z]]=[[x,y],z]+(-1)^{|x||y|}[y,[x,z]]$$
for $x,y,z\in \LL U$ and Leibniz rule: 
$$ \partial [x,y]=[\partial x,y]+(-1)^{|x|}[x,\partial y].$$
Note that  $L(X)$ determines the rational homotopy type of $X$, namely $||L(X)||\simeq X_0$.
In particular,  
there are isomorphisms
$$ \tilde{H_n}(X;\Q )\cong H_{n-1}(U,\partial_1) \mbox{\ \ \ and\ \ } \pi_*(\Omega X)_{\Q} \cong H_{*}(\LL U,\partial ),$$
where $\partial_1:U\to U$ is the linear part of $\partial$. Here the second is an isomorphism as graded Lie algebras. 
Refer to \cite{FHT} for detail.\\



\noindent
{\it Proof of Lemma \ref{equi}}.
Recall the DGA-diagram of \S 1.
Then  $D_2=\psi\circ D_1\circ \psi^{-1}$.
Therefore there is a DGL-isomorphism  $\phi$  given by $\phi (\sigma )=\psi \circ \sigma \circ \psi^{-1}$
and 
$$\xymatrix{ Der  (\Lambda V\otimes \Lambda W,D_1),\partial_1\ar[r]^{\ \ \  \ \ \ \ b_{f_1}}\ar[d]_{\phi}& Der  (\Lambda V)\ar@{=}[d]\\
Der  (\Lambda V\otimes \Lambda W,D_2),\partial_2\ar[r]^{\ \ \ \ \ \ \ b_{f_2}}& Der  (\Lambda V)}$$
is DGL-commutative
since $\psi |_{\Lambda V} =id_{\Lambda V}$.
In particular, we can check 
$\partial_2\circ \phi=\phi \circ \partial_1$
by $$\partial_2 \phi (\sigma)=D_2\psi\sigma\psi^{-1}-(-1)^i(\psi\sigma\psi^{-1})D_2=\psi D_1\sigma\psi^{-1}-(-1)^i\psi\sigma D_1\psi^{-1}
$$
$$=\phi (D_1\sigma-(-1)^i\sigma D_1)=\phi  \partial_1(\sigma)$$
for $\sigma\in Der_i  (\Lambda V\otimes \Lambda W,D_1)$.
Similarly we have   $\phi ([\sigma,\tau ])=[\phi (\sigma ),\phi (\tau)]$. 
\hfill\qed\\

\noindent
{\bf Convention.} 
For a DGA-model $(\Lambda V,d)$ the symbol $(v,f)$ means the {\it elementary derivation} that takes a generator $v$ of $V$ to an  element $f$ of $\Lambda V$
and the other generators to $0$.
Note that $|(v,f)|=|v|-|f|$.\\

\noindent
{\it Proof of Proposition \ref{sep}}. Let  $M(Y)=(\Lambda V,d)\to (\Lambda V\otimes \Lambda W,D)$ be the model of $f$.\\
(if) When $\min  W\geq \max  V$,
there is a decomposition of vector spaces$$ Der  (\Lambda V\otimes \Lambda W)=Der  (\Lambda V)\oplus 
Der  ( \Lambda W, \Lambda V\otimes \Lambda W)$$
from degree arguments.
Then  there is a DGL-map $b_f:Der  (\Lambda V\otimes \Lambda W,D)\to Der(\Lambda V,d)  $
by  $b_f(\sigma_1)=\sigma_1$ and $b_f (\sigma_2 )=0$ for $\sigma=\sigma_1+\sigma_2$ with   $\sigma_1\in Der  (\Lambda V)$
and $\sigma_2\in Der  ( \Lambda W, \Lambda V\otimes \Lambda W)$.
In particular, 
it preserves the differential since $b_f(\tau (\sigma_1))=0$
when 
 $\partial_X(\sigma_1)=\partial_Y (\sigma_1)+\tau (\sigma_1)$
for $\tau (\sigma_1)
\in Der  ( \Lambda W, \Lambda V\otimes \Lambda W)$.\\
(only if) Suppose that  $\min W< \max V$.
There are elements $w\in W$ and $v\in V$
with $|w|<|v|$.
Then $b_f$ is not a DGL-map since
$$0\neq b_f(v,1)=b_f([(w,1),(v,w)])=[b_f(w,1),b_f(v,w)]=[0,0]=0$$
from the definition of $b_f$.
\hfill\qed\\

\begin{exmp}
Let $f:S^{2n}\to K(\Z, 2n)$ be the natural inclusion.
Then the homotopy fibre is $S^{4n-1}$ and therefore $f$ is  $\pi_{\Q}$-separable.
Thus   $b_f :Der (\Lambda (x,y),D)\to Der (\Lambda (x),0)$
with $|x|=2n$, $|y|=4n-1$, $Dx=0$ and $Dy=x^2$ is given by
$b_f((x,1))=(x,1)$ and $b_f((y,1))=b_f((y,x))=0$.
Refer Theorem \ref{one}.
\end{exmp}

\begin{exmp}
Consider the case that $f$ is not $\pi_{\Q}$-separable (not $\Q$-w.t.).
Let 
$f:S^7\to S^4$
be the Hopf map. Then the model  is given by
$$M(S^4)=(\Lambda (x,y),d)\to (\Lambda (x,y,z),D)\simeq M(S^7)$$
with $|x|=4$, $|y|=7$, $|z|=3$, $dx=Dx=0$, $Dy=dy=x^2$, $Dz=x$, $Dy=x^2$.
Then the bases of derivations are given as \\
 
\begin{center}
\begin{tabular}{l|c}

\multicolumn{1}{c|}{$n$} &
\multicolumn{1}{c}{$Der_n (\Lambda (x,y,z),D)$} \\

\hline
$7$ & $(y,1)$\\
\hline
$4$ &  $(x,1)$ $(y,z)$\\
\hline
$3$ &  $(y,x)$ $(z,1)$\\
\hline
$1$ &  $(x,z)$\\
\hline
\end{tabular}
\ \ \ \ \ \ \ \ \ 
\begin{tabular}{l|c}

\multicolumn{1}{c|}{$n$} &
\multicolumn{1}{c}{$Der_n (\Lambda (x,y),d)$} \\

\hline
$7$ & $(y,1)$\\
\hline
$4$ &  $(x,1)$\\
\hline
$3$ &  $(y,x)$\\
\hline
$1$ &  \\
\hline 
\end{tabular}
\end{center}
, where $H_*(Der  (\Lambda (x,y))=\Q \{ (y,1)\}$.
By degree reason, any  DGL-map
$$\psi :(Der (\Lambda (x,y,z),D)\to (Der (\Lambda (x,y),d)$$
is given by 
$\psi (y,1)=a_1(y,1)$,   
$\psi (x,1)=a_2(x,1)$, 
$\psi (y,z)=a_3(x,1)$, 
$\psi (y,x)=a_4(y,x)$, 
$\psi (z,1)=a_5(y,x)$ 
and $\psi_f (x,z)=0$
for some $a_i\in\Q$.

From $(x,1)=[(z,1),(x,z)]$ and $(y,z)=[(x,z),(y,x)]$
we have $a_2=0$
and  $a_3=0$, respectively.
Then from $2(y,1)=[(z,1),(y,z)]+[(x,1),(y,x)]$, we obtain $a_1=0$.
 Thus 
 $||\psi ||$ is homotopic to the constant map.
\end{exmp}

\noindent
{\it Proof of Theorem  \ref{section}}.
The map
$\pi_n(\psi_f):\pi_n(aut_1X)\to \pi_n(aut_1Y)$
 is given by  $\pi_n(\psi_f)([\sigma ])=[\tau]:=[f\circ \sigma\circ (s\times 1_{S^n})]$
  in the following homotopy commutative diagram:
{\small $${\small \xymatrix{
X\times S^n\ar[dd]_{\sigma  }&&Y\times S^n\ar[ll]_{s\times id_{S^n}}\ar@{.>}[dd]^(.3){\tau}&\\
&X\ar[lu]^{i_X}\ar[dl]_g &&Y\ar[ll]^{s\ \ \ \ \ \ \ \ }\ar@{=}[dl]\ar[lu]_{i_Y}\\
X\ar[rr]_f&&Y&
}}$$} 
from adjointness.
That is 
the pointed homotopy classes of maps $S^n\to aut_1X=map(X,X;id_X)$ are in bijection with the homotopy 
classes of those maps $X\times  S^n\to X$  that composed with the inclusion $i_X:X\to X\times S^n$ yield the identity \cite[p.43-44]{Sa}. 
Let  $M(Y)=(\Lambda V,d)\to (\Lambda V\otimes \Lambda W,D)$ be the model of $f$.
There is a chain map $$c_f:Der  (\Lambda V\otimes \Lambda W,D)\to Der(\Lambda V,d)  $$ 
given by $c_f (\sigma )={\rm proj}_V\circ \sigma$.
It is well-defined, i.e., 
$\partial_Y\circ c_f=c_f\circ \partial_X$,  from  $DW\subset \Lambda V\otimes \Lambda^+W$ \cite{Th} since   it  admits  a section.
Notice that  $\pi_n(\psi_f)_{\Q}([\sigma_0 ])=H_n(c_f)([\sigma ])=[\tau]$ by identifying the maps $\sigma$ and $\tau$ 
to their induced derivations of $\Lambda V\otimes \Lambda W$ and $\Lambda V$, respectively \cite[p.313]{S}(cf.\cite[Proposition 11]{Sa}).
(However $c_f$ is not a DGL-map in general.)
If $f$ is a $\pi_{\Q}$-separable map,
$c_f=b_f$ as DGL-maps and then $\pi_n(\Omega a_f)$  is identified with 
$$\pi_n(\psi_f)_{\Q}:\pi_n(aut_1X)_{\Q}\cong \pi_n(\Omega Baut_1X)_{\Q}
\overset{H_n(b_f)}{\longrightarrow} \pi_n(\Omega Baut_1Y)_{\Q} \cong \pi_n(aut_1Y)_{\Q}$$
for any $n$.
Thus
we obtain  $\Omega a_f\simeq (\psi_f)_0$. 
\hfill\qed\\


The following is obvious from the definition of $b_f$ and useful:
\begin{claim}\label{rea}
For any $\pi_{\Q}$-separable  map $f:X\to Y$, we have
$b_f(C)=0$ and $b_f\mid_{Der(\Lambda V)}=id_{Der(\Lambda V)}$ for $Der  (\Lambda V\otimes \Lambda W)=C\oplus Der(\Lambda V)$.
\end{claim}

\noindent
{\it Proof of Theorem \ref{xy}}.
Let $X\to E_{\infty}^X\overset{p_{\infty}^X}{\to} Baut_1X$ and $Y\to E_{\infty}^Y\overset{p_{\infty}^Y}{\to}  Baut_1Y$ be the universal fibrations of $X$ and $Y$, respectively.
Let  $C^*(Der (\Lambda V))\otimes \Lambda V,D_Y$ be the DGA-model of $E_{\infty}^Y$
and $C^*(Der (\Lambda V\otimes \Lambda W)) \otimes \Lambda V\otimes \Lambda W,D_X$ be the DGA-model of $E_{\infty}^X$.
For 
a $\pi_{\Q}$-separable map $f:X\to Y$,
there exists a DGA-inclusion map $\psi$ such that 
$${\small  
\xymatrix{ C^*(Der (\Lambda V))\ar[d]_{C^*(b_f)}\ar[r]^{}&
  C^*(Der (\Lambda V))\otimes \Lambda V,D_Y\ar@{.>}[d]_{{\psi }}\ar[r]^{}& \Lambda V,d\ar[d]^{i}\\
C^*(Der (\Lambda V\otimes \Lambda W))\ar[r]^{}& 
{ C^*(Der (\Lambda V\otimes \Lambda W) \otimes \Lambda V\otimes \Lambda W,D_X}\ar[r]^{} &  \Lambda V\otimes \Lambda W,D
\\
}}$$
is commutative from the universality.
Indeed,  
$C^*(Der (\Lambda V))\otimes \Lambda V,D_Y$ is a sub-DGA of 
$C^*(Der (\Lambda V\otimes \Lambda W)) \otimes \Lambda V\otimes \Lambda W,D_X$ from
 Claim \ref{rea} and  $\max V\leq \min W$.
Thus there is a map $\tilde{a}_f:=|\psi |:  (E_{\infty}^X)_0 \to (E_{\infty}^Y)_0$
such that $(p_{\infty}^Y)_0\circ \tilde{a}_f=a_f\circ (p_{\infty}^X)_0$.
Since  
$p'$ is the pull-back of  $(p_{\infty}^Y)_0$ by $h'$,
there exists a map $k:E\to E'$ such that
$$ {\small
\xymatrix
{ E \ar[dd]_{p} \ar@{.>}[rd]^{k } \ar[rr]^{\tilde{h}} && (E_{\infty}^X)_0 \ar'[d][dd]^(.3){(p_{\infty}^X)_0} \ar[rd]^{\tilde{a}_f} \\
&E' \ar[dd]_(.3){p'} \ar[rr] ^(.3){ \tilde{h'} }  &&  (E_{\infty}^Y)_0 \ar[dd]^{(p_{\infty}^Y)_0} \\
B_0 \ar'[r] [rr]_(.3){h}    \ar[rd]_{g} && (Baut_1X)_0  \ar[rd]^{a_f} \\
& B'_0 \ar[rr]^{h' }  &&  (Baut_1Y)_0
}}
$$
is commutative from the universality
since $h'\circ g\circ p=a_f\circ h\circ p=a_f\circ (p_{\infty}^X)_0\circ \tilde{h}=(p_{\infty}^Y)_0\circ \tilde{a}_f\circ \tilde{h}$.
\hfill\qed\\


\begin{exmp}\label{real}
Let $X=K(\Q ,n)\times K(\Q, 2n)$
and $Y=K(\Q, n)$ for some even integer $n$.
Then $M(X)=\Lambda (x,y),0$ and $M(Y)=\Lambda (z),0$ with $|x|=|z|=n$ and $|y|=2n$.
Let a map $f:X\to Y$ be given by  $M(f):\Lambda (z)\to \Lambda (x,y)$ with $M(f)(z)=x$.
The homotopy fibration of any  $\pi_{\Q}$-separable map 
 is given by $\Lambda (z),0\to \Lambda (z,y),0\cong \Lambda (x,y),0$ 
from the degree reason.
Therefore the DGL-map
$\psi :Der \Lambda (x,y)\to Der \Lambda (z)$
such that $\psi ((y,x))=\psi((x,1))=(z,1)$ is not DGL-homotopic to $b_f$ from Claim \ref{rea}.

Let $h:S^{n+1}_0\to  (Baut_1X)_0$ and  $h':S^{n+1}_0\to  (Baut_1Y)_0$
be  given by $L(h): \LL (u)\to Der (\Lambda (x,y))$ with $|u|=n$, $L(h)(u)=(y,x)$
and $L(h'): \LL (u)\to Der (\Lambda (z))$ with $L(h')(u)=(z,1)$, respectively.
Then the commutative diagram 
 $${\small  
\xymatrix{ 
S^{n+1}_0\ar[d]_{{=}}\ar[r]^{h\ \ }& (Baut_1X)_0\ar[d]^{||\psi ||}\\
{ S^{n+1}_0}\ar[r]^{h'\ \ } & (Baut_1Y)_0
\\
}}$$
does not induce   
a map between  total spaces $f':E\to E'$ such that 
$${\small  
\xymatrix{ X_0\ar[d]_{f_0}\ar[r]^{}&
 E\ar@{.>}[d]_{{ f'}}\ar[r]^{p}& S^{n+1}_0\ar[d]^{=}\\
Y_0\ar[r]^{}& 
{ E'}\ar[r]^{p'} & S^{n+1}_0
\\
}}$$
is homotopy commutative.
Indeed,  there does not exist 
a DGA-map  $h:\Lambda (v,z),D'\to \Lambda (v,x,y),D$ 
with $D'z=v$, $Dy=vx$ and $Dx=0$ such that 
$${\small  
\xymatrix{ \Lambda v,0\ar[d]_{=}\ar[r]^{}&
 \Lambda (v,z),D'\ar@{.>}[d]_{{ h}}\ar[r]^{p}& \Lambda (z),0\ar[d]^{M(f)}\\
\Lambda v,0\ar[r]^{}& 
{ \Lambda (v,x,y),D}\ar[r]^{p'} &  \Lambda (x,y),0
\\
}}$$
where $|v|=n+1$
is homotopy commutative
since  $h$ can  not be a DGA-map from  $Dh(z)=0$ but $hD'(z)=v$. 
\end{exmp}

\begin{rem} Recall that a map $f:X\to Y$ is said to be a {\it Gottlieb map} \cite{Y}
if its homotopy group map  induces
the map between their Gottlieb groups \cite{Gott} $f_{\sharp}:G_*(X)\to G_*(Y)$.
For example, if a map $f$ admits a section, it is a Gottlieb map. 
Since elements of the rational Gottlieb group $G_*(X)_{\Q}$ is described by certain  derivations of $M(X)$ \cite{FH},
we see  that
if a map $f:X\to Y$ is $\pi_{\Q}$-separable, it is a rational Gottlieb map.
\end{rem} 
  

\section{When does $a_f$ admit a section ?}
Let  $f:X\to Y$ be  a $\pi_{\Q}$-separable map with homotopy fiber $F_f$ and 
$ Der  ( \Lambda W, \Lambda V\otimes \Lambda W)$
the  sub-DGL of $Der  (\Lambda V\otimes \Lambda W)$
 restricted to derivations out of  $\Lambda W$.

\begin{prop}\label{Tanre}Let $F_{a_f}$ be the homotopy fiber of $a_f$. 
Then the DGL-model of the fibration $\chi_f :F_{a_f}\overset{j}{\to} (Baut_1X)_0\overset{a_f}{\to} (Baut_1Y)_0$ is given by
$$ Der  ( \Lambda W, \Lambda V\otimes \Lambda W)\overset{incl.}{\to} Der  (\Lambda V\otimes \Lambda W)\overset{b_f}\to Der  (\Lambda V).$$
\end{prop}
\begin{proof}Since $b_f$ is surjective and $ Der  ( \Lambda W, \Lambda V\otimes \Lambda W)$
is  ${\rm Ker}\ b_f$, it follows from
 \cite[VI .1.(3) Proposition]{T}.
\end{proof}

Let $L(F)=\oplus_{i>0}L(F)_i$
be the degree decomposition of a DGL-model of a space $F$.

\begin{thm}\label{fiber}
 $L(F_{a_f})_n=
\oplus_{i-j=n}Der_i  ( \Lambda W)\otimes H^j(\Lambda V)$.
\end{thm}
\begin{proof}
A chain-map
$\rho :Der_i  ( \Lambda W)\otimes H^j(\Lambda V)\to Der_i  (\Lambda W,  \Lambda W\otimes (\Lambda V)^j) $
is given by $\rho (\sigma \otimes [f])(w):=(-1)^{|w|j}\sigma (w)\cdot f$
 induced by an inclusion $H^j(\Lambda V)\to (\Lambda V)^j$.
It is quasi-isomorphic, i.e., 
 there is a decomposition 
$Der  (\Lambda W,  \Lambda W\otimes \Lambda V)=(Der  ( \Lambda W)\otimes H^*(\Lambda V))\oplus C$
for a complex $C$ of derivations 
with  ${H}_*(C)=0$.
\end{proof}


The rational homotopy exact sequence of the strictly induced fibration $\chi_f$:
$$\ \ \ \ \  \cdots \overset{j_{\sharp}}{\to} 
\pi_{n+2}(Baut_1X)_{\Q}\overset{{a_f}_{\sharp}}\to \pi_{n+2}(Baut_1Y)_{\Q} \overset{\delta_f}{\to} $$
$$ \pi_{n+1}(F_{a_f})_{\Q}\overset{j_{\sharp}}{\to} 
\pi_{n+1}(Baut_1X)_{\Q}\overset{{a_f}_{\sharp}}\to \pi_{n+1}(Baut_1Y)_{\Q} \overset{\delta_f}{\to} \cdots$$
is equivalent to  the homology exact sequence:
$$\ \ \ \ \ \ \ \ \ \ \ \ \ \ \ \ \ \ \ \ \ \cdots \to 
H_{n+1}(Der  (\Lambda V\otimes \Lambda W))\overset{{b_f}_*}\to H_{n+1}(Der  (\Lambda V)) \overset{\delta_f}{\to} $$
$$ H_n(Der  ( \Lambda W, \Lambda V\otimes \Lambda W))\to 
H_n(Der  (\Lambda V\otimes \Lambda W))\overset{{b_f}_*}\to H_n(Der  (\Lambda V)) \overset{\delta_f}{\to} \cdots$$

Then we have the following from an ordinary chain complex property:
\begin{claim}\label{deltaf}
The connecting map  $\delta_f$ is   given by $\delta_f ([\sigma ])=[\tau ]$ when  $\partial_X (\sigma )=\tau$
for a $\partial_Y$-cycle  $\sigma$ of $Der  (\Lambda V)$
and a $\partial_X$-cycle  $\tau$ of $Der  (\Lambda W, \Lambda V\otimes \Lambda W)$. 
\end{claim}
Recall that the following implications hold for a general fibration $\chi: F\to E\overset{p}{\to} B$ of simply connected spaces:
$$ \chi \mbox{ is fibre-trivial }\Rightarrow\  p \mbox{ admits a section }\Rightarrow\  \chi \mbox{ is weakly trivial }\Leftrightarrow\  
\delta=0 .\ \ \ \ \ (*)$$
Here   $\delta :\pi_*(B)\to \pi_{*-1}(F)$ is the connecting map of the homotopy exact sequence for $\chi$.
The following may be  a characteristic  phenomenon in   our fibration $\chi_f$.

\begin{prop}\label{delta} 
$a_f$ admits  a  section if and only if $\delta_f =0$.
\end{prop}
\begin{proof} 
(if)   Let the DGA-model of the fibration $\chi_f$ be given as the commutative diagram:
$${\small \xymatrix{
\Lambda U,d\ar[d]^{\simeq}_{\rho_2\mid_{\Lambda U}}\ar[r]&\Lambda U\otimes \Lambda Z,D_2\ar[d]^{\simeq}_{\rho_2}\ar[r]& \Lambda Z,\overline{D}_2\ar@{=}[d]\\
C^*(Der\Lambda V)\ar[r]\ar@{=}[d]&C^*(Der\Lambda V)\otimes \Lambda Z,D_1\ar[d]^{\simeq}_{\rho_1}\ar[r]& \Lambda Z,\overline{D}_1\ar[d]^{\simeq}\\
C^*(Der\Lambda V)\ar[r]_{ C^*(b_f)\ \ \ }&\ \ C^*(Der(\Lambda V\otimes \Lambda W)),D\ar[r]_{C^*(incl.)}&\ \ C^*(Der( \Lambda W, \Lambda V\otimes \Lambda W)),D',
}}$$
where  $M(Baut_1Y)\cong (\Lambda U,d)$ with $U^{n+1}= H_n(Der  (\Lambda V))$ and
 $M(F_{a_f})\cong ( \Lambda Z,\overline{D}_2)$ with  $Z^{n+1}=  H_n(Der  ( \Lambda W, \Lambda V\otimes \Lambda W))$.
From the assumption
 $\chi_f$ is weakly equivalent, i.e.,
$M(Baut_1X)\cong (\Lambda U\otimes \Lambda Z,D_2)$.
 Let $D=d_1+d_2$ in \S 2.
Then $$ H^{n+1}( C^*(Der(\Lambda V\otimes \Lambda W)),d_1)\cong H_n(Der(\Lambda V\otimes \Lambda W))= U^{n+1}\oplus Z^{n+1}.  $$
Notice that 
$(w,h)
\not\in [Der  (\Lambda V),Der  (\Lambda V)]$
for any $w\in W$ and $h\in \Lambda V\otimes \Lambda W$, where 
$[\ ,\ ]$ is the Lie bracket.
That means 
$$d_2(s^{-1}(w,h)^*)\in I( C^*(Der( \Lambda W, \Lambda V\otimes \Lambda W))) \ \ \ \ \ \ \ \ \ (**)$$ 
Here $I(S )$ is the ideal generated by a basis of a vector space $S$. 
Let  $\sigma$ be  a $\partial_X$-cycle of $Der( \Lambda W, \Lambda V\otimes \Lambda W)$.
Then $[s^{-1}\sigma^* ]\in H_*( C^*(Der( \Lambda W, \Lambda V\otimes \Lambda W)),d_1')\cong Z$
for $D'=d_1'+d_2'$ in \S 2.
Since $$ \rho_1(D_1([s^{-1}\sigma^* ]))\sim d_2(s^{-1}\sigma^*)\ \ \ ;D{\mbox{-cohomologous}},$$
we have  $ D_1([s^{-1}\sigma^* ])\in I(Z)$ by $(**)$,
i.e., 
$D_1(Z)\subset  C^*(Der\Lambda V)\otimes \Lambda^+ Z$.
By $\rho_2$,
 $D_2(Z)\subset  \Lambda U\otimes \Lambda^+ Z$.
Then we have done 
 from  \cite{Th}.\\
(only if) It holds from the above implications $(*)$.
\end{proof}


\begin{thm} If  a $\pi_{\Q}$-separable map
$f:X\to Y$  is rationally fibre-trivial (i.e.,  $X_0\sim (F_f)_0\times Y_0$),
 $a_f$ admits  a  section.
\end{thm}
\begin{proof}
From the assumption and Claim \ref{deltaf},
we have $\delta_f=0$.
Then it holds from Proposition \ref{delta}.
\end{proof}
Refer \cite[page 292]{NY} for related topics.
Conversely, when $Y$ is an odd-sphere, 
\begin{thm}\label{sec}
If  a $\pi_{\Q}$-separable map $f:X\to Y=S^{2n+1}$ is 
not rationally  fibre-trivial,  
$a_f$ does not admit a section.
Furthermore  ${a_f}\sim *$(the constant map).
\end{thm}
\begin{proof}
Let $M(S^{2n+1})=(\Lambda v,0)$.
Since 
there exists an element $w\in W$ such that $Dw\in \Lambda v\otimes \Lambda^+W$ from the assumption,
$\partial_X(v,1)=\pm (w,\partial Dw/\partial v)+\cdots \neq 0$
in $Der (\Lambda W)$.
From Claim \ref{deltaf} $\delta_f$ is injective  since $\delta_f([(v,1)])=[\pm (w,\partial Dw/\partial v)+\cdots ]\neq 0$.
Then the former  holds from Proposition \ref{delta}.
Furthermore, from the homotopy exact sequence,   we have ${a_f}_{\sharp}=0$.
Thus the latter holds.
\end{proof}

\begin{exmp}\label{counter}
(1)
Let $S^5\times S^7\to X\to Y=S^3$ be a  non-(fibre-)trivial  $\pi_{\Q}$-separable  fibration given by
the model  
$$(\Lambda (v_1),0)\to (\Lambda (v_1, w_1,w_2),D) \to (\Lambda (w_1,w_2),0) $$
with $|v_1|=3$, $|w_1|=5$, $|w_2|=7$, 
 $Dw_1=0$ and  $Dw_2=v_1w_1$.
Then $a_f$ does not admit a section from Theorem  \ref{sec}. 
Indeed $\delta_f:H_{3}(Der  (\Lambda v')) \overset{}{\to}  H_2(Der(\Lambda (w_1,w_2), \Lambda (v_1,w_1,w_2)))$  
is non-trivial from
$\delta_f ([(v_1,1)])=[(w_2,w_1)]\neq 0$.





(2) Let $S^5\times S^7\to X'\to Y'$ be a  non-(fibre-)trivial  $\pi_{\Q}$-separable fibration given by
the model  
$$(\Lambda (v_1,v_2,v_3),d_{Y'})\to (\Lambda (v_1,v_2,v_3, w_1,w_2),D') \to (\Lambda (w_1,w_2),0) $$
with $|v_1|=|v_2|=3$, $|v_3|=5$, $|w_1|=7$, $|w_2|=9$, 
$d_{Y'}(v_1)=d_{Y'}(v_2)=0$, $d_{Y'}(v_3)=v_1v_2$, $D'w_1=0$ and  $D'w_2=v_1w_1$.
Then $a_f$ admits a section from Proposition \ref{delta} since  $\delta_f ([(v_3,1)])= 0$
for $H_*(Der(\Lambda (v_1,v_2,v_3)))=\Q \{ [(v_3,1)]\}$.
However  $\chi_f$ is not trivial
from $[(v_3,1),(w_2,v_2v_3)]=(w_2,v_2)$.
Indeed, then 
$${\mathcal{D}}(s^{-1}(w_2,v_2)^*)=d_2(s^{-1}(w_2,v_2)^*)=s^{-1}(v_3,1)^*\cdot s^{-1}(w_2,v_2v_3)^*$$
for $(C^*(Der(\Lambda (v_1,v_2,v_3, w_1,w_2)),{\mathcal D})$ with ${\mathcal{D}}=d_1+d_2$.
Refer the proof of Proposition \ref{delta}.
\end{exmp}

From Claim \ref{rea}, we obtain

\begin{lem}\label{lemm}
Let $X$ be a pure space
with $M(X)=( \Q [x_1,\cdots ,x_m]\otimes \Lambda (y_1, \cdots ,y_n),d)$ with 
$\max \{ |x_1|, \cdots , |x_m|\}<\min \{ |y_1|, \cdots ,  |y_n|\}$.
If $f:X\to \Pi_{i=1}^m K(\Q, |x_i|)$
is the rational principal fibration given by $M(f): ( \Q [x_1,\cdots ,x_m],0)\to M(X)$ with 
$M(f)(x_i)=x_i$ for all $i$.
Then $a_f\sim *$ if and only if $S=0$ 
for $H_*(Der M(X))=S\oplus T$
with $S\subset Der ( \Q [x_1,\cdots ,x_m])$
and $T\cap Der  (\Q [x_1,\cdots ,x_m])=0$.
\end{lem}

\begin{thm}\label{one}
Let $X$ be an $F_0$-space with $\max \{ |x_1|, \cdots , |x_n|\}<\min \{ |y_1|, \cdots ,  |y_n|\}$. 
Then $a_f\sim *$ for the map 
 $f:X\to Y:=\Pi_{i=1}^n K(\Q, |x_i|)$
of above 
if and only if 
$X$
 satisfies the  Halperin's conjecture.
\end{thm}
\begin{proof}
It follows from Lemma \ref{lemm} since the  Halperin's conjecture  is equivalent to that 
$Der H^*(X;\Q )=Der (\Q [x_1,\cdots ,x_n]/(f_1,\cdots ,f_n))=0$ \cite{M}.
\end{proof}

\noindent
{\it Proof of Theorem \ref{two}}. 
Let $M(Y)=(\Lambda V,d)=( \Q [x_1,\cdots ,x_n]\otimes \Lambda (y_1, \cdots ,y_n),d)$ with $dx_i=0$ and $dy_i=f_i$
for $i=1,..,n$.\\
(if) Let $M(Y)=(\Lambda V,d)\to (\Lambda V\otimes \Lambda W,D)$
be the model of $f$. 
From the regularity of $f_1,\cdots ,f_n$, $Im D\subset  \Q [x_1,\cdots ,x_n]\otimes \Lambda W$.
Thus 
$$\partial_X(x_i,h_i)=\sum_{j=1}^n(y_j,(\partial f_j/\partial x_i)\cdot h_i)+\theta \ \ \mbox{and} \ \ \partial_X(y_i,h_i)=0\ \ \ \ (i=1,\cdots ,n)$$
for  any $h_i\in  \Q [x_1,\cdots ,x_n]$ with suitable degree  and some $\theta \in Der ( \Lambda W, \Lambda V\otimes \Lambda W)$. 
Then we have $\delta_f=0$ from Claim \ref{deltaf} since
 $H_{even}(DerM(Y))=0$ \cite{M} from the assumption.
Then $a_f$ admits a section from Proposition  \ref{delta}.
Furthermore the Lie bracket decomposition of
an element of  $ Der  ( \Lambda W, \Lambda V\otimes \Lambda W)$
does not have an element of $Der (\Lambda V)$ as a factor from Theorem \ref{fiber}
since $Der H^*(Y;\Q )=0$ \cite{M} again. 
Thus we have $D_2=d\otimes 1\pm 1\otimes \overline{D}_2$
for the Sullivan minimal model $(\Lambda U,d)\to (\Lambda U\otimes \Lambda Z,D_2)\to (\Lambda Z,\overline{D}_2)$ of $\chi_f$ 
(in the proof of Proposition \ref{delta}).\\
(only if) Suppose there is a non-zero element  $[\sum_{i}(x_i,h_i)+\sum_{j}(y_j,g_j) ]\in H_{2m}(DerM(Y))$ for
$h_i\in  \Q [x_1,\cdots ,x_n]$, $g_j\in \Lambda V$ and  some $m$.
Let $S^a\times S^b\to X\overset{f}{\to} Y$
be a rational fibration of the model:
$$  (\Lambda V,d)\to (\Lambda V\otimes \Lambda (w_1,w_2),D)\to (\Lambda (w_1,w_2),0)$$
where $|w_1|=a$ and $|w_2|=b$ are odd with $b-a=|x_k|-1$ for some $k$,
$h_k$ is not $d_Y$-exact, $Dw_1=0$ and 
$Dw_2=x_kw_1$.
Then 
$$\delta_f( [\sum_{i}(x_i,h_i)+\sum_{j}(y_j,g_j)  ])=[(w_2,h_kw_1)]\neq 0$$
for $\delta_f:H_{2m}(Der \Lambda V)\to H_{2m-1}(Der (\Lambda (w_1,w_2), \Lambda V\otimes \Lambda (w_1,w_2))$.
In particular,  $\chi_f$ is not fibre-trivial.
\hfill\qed\\


\begin{exmp} 
Let $Y$ be the homogeneous space $SU(6)/SU(3)\times SU(3)$.
Then $Y$ is a pure space but not an $F_0$-space since $\rank \ SU(6)=5>4=\rank \ (SU(3)\times SU(3))$.
Let
$\xi :S^{11}\times S^{23}\to X\overset{f}{\to} Y$ be a fibration whose relative model is given as
$$(\Lambda (x_1,x_2,y_1,y_2,y_3),d_Y)\to  (\Lambda (x_1,x_2,y_1,y_2,y_3)\otimes \Lambda (w_1,w_2),D)\to  (\Lambda (w_1,w_2),0)
$$
where $|x_1|=4$,   $|x_2|=6$,   $|y_1|=7$,   $|y_2|=9$,   $|y_3|=11$,    $|w_1|=11$,   $|w_2|=23$,  
$d_Yy_1=x_1^2$, $d_Yy_2=x_1x_2$, $d_Yy_3=x_2^2$, $Dw_1=0$ and $Dw_2=(x_1y_2-x_2y_1)w_1$.
Then $\partial_X((y_1,1))=(w_2,x_2w_1)$, i.e., $\delta_f [(y_1,1)]=[(w_2,x_2w_1)]\neq 0$.
In particular $\chi_f$ is not trivial.
Refer \cite[Example 1.14(2)]{NY} for the Sullivan minimal model of $Baut_1Y$.
\end{exmp}






In this section finally we mention about a heredity property for a pull-back:
\begin{thm}\label{pull} Let $f':X'\to Y'$ be the  pull-back of a map $f:X\to Y$ by a map $g:Y'\to Y$ with a rational section.
Suppose that both $f$ and $f'$ are $\pi_{\Q}$-separable.
If 
$a_f$ admits a section, then  $a_{f'}$ does so.
\end{thm}
\begin{proof}
From \cite[Proposition 15.8]{FHT}, the Sullivan model of the pull-back diagram:
$$\xymatrix{F_{f'}\ar[d]_{i'}\ar@{=}[r]& F_f\ar[d]^i\\
X'\ar[d]_{f'}\ar[r]^{g'}& X\ar[d]^f\\
Y'\ar[r]_g&Y
}$$
is given as 
$${\small \xymatrix{\Lambda W,\overline{D'} & \Lambda W,\overline{D}\ar@{=}[l]\\
\Lambda V\otimes \Lambda U\otimes \Lambda W,D'\ar[u]& \Lambda V\otimes \Lambda W,D\ar[u]\ar[l]\\
\Lambda V\otimes \Lambda U,d'\ar[u]&\ar[l]\Lambda V,d\ar[u]
}}$$
where $D'W=DW\subset \Lambda V\otimes \Lambda W$.
Then, from Claim \ref{deltaf},  the following is commutative:
$$\xymatrix{
H_*(Der(\Lambda V))\ar[r]^{\delta_f\ \ \ \ \ \ \  \ \ }&H_{*-1}(Der (\Lambda W, \Lambda W\otimes H^*(Y;\Q))\ar[d]^{H_{*-1}(Der(g^*))}\\
H_*(Der(\Lambda V\otimes \Lambda U))\ar[r]^{\delta_{f'}\ \ \ \ \ \ \ \ \ \ }\ar[u]^{H_*(c_g)}&H_{*-1}(Der (\Lambda W,\Lambda W\otimes H^*(Y';\Q)), }
$$
where $c_g$ is same as $b_g$ as a chain map (see the proof of Theorem  \ref{section}).
Here $Der(g^*): Der (\Lambda W, \Lambda W\otimes H^*(Y;\Q))\to Der (\Lambda W,\Lambda W\otimes H^*(Y';\Q))$
is given  by $Der(g^*)((w,w'\otimes y))=(w,w'\otimes g^*(y))$.
Then 
$\delta_{f'}=0$ if $\delta_f=0$.
Thus it  follows from Proposition \ref{delta}.
\end{proof}

Remark that $\chi_{f'}$ is not a pull-back of $\chi_f$.
Moreover the converse of  Theorem  \ref{pull}  is false in general.
Indeed, we see in Example \ref{counter}
that the fibration of (2) is the pull-back of (1) by a map $g:Y'\to Y=S^3$ with $M(g)(v_1)=v_1$.

\section{The obstruction class for a lifting}

Let $L(B)=(L(B),\partial_B)$ be the Quillen model of a simply connected CW complex $B$ of finite type.
Then $L(B\cup_{\alpha} e^N)$ is given by  $L(B)\coprod \LL(u),\partial_{\alpha}$ 
where  $|u|=N-1$, $\partial_{\alpha}\mid_{L (B)}=\partial_B$ and $\partial_{\alpha}(u)\in L(B)$ \cite[Proposition III.3.(6)]{T}.
\begin{thm}\label{obs}  For a $\pi_{\Q}$-separable map   $f:X\to Y$, let 
$$\xymatrix{B\ar[r]^{h_X\ \ \ }\ar[d]_i&  (Baut_1X)_0\ar[d]^{a_f}\\
B\cup_{\alpha} e^N \  \ar[r]_{h_Y}&\ (Baut_1Y)_0
}$$
be a  commutative diagram.
Then there is a lift $h$ such that
$$\xymatrix{B\ar[r]^{h_X\ \ \ }\ar[d]_i&  (Baut_1X)_0\ar[d]^{a_f}\\
B\cup_{\alpha} e^N \  \ar[r]_{h_Y}\ar@{.>}[ur]^{{h}}&\ (Baut_1Y)_0
}$$
is  commutative if and only if 
$${\mathcal O}_{\alpha}(h_X, h_Y):=[\tau (h_Y(u))-h''_X(\partial_{\alpha}(u)) ]=0$$
in $H_{N-2}(Der  ( \Lambda W, \Lambda V\otimes \Lambda W))=\pi_{N-1}(F_{a_f})_{\Q}$ for  the DGL-commutative  diagram
$$\xymatrix{L(B)\ar[r]^{h_X\ \ \ \ \ \  }\ar[d]_i& Der  (\Lambda V\otimes \Lambda W),\partial_X\ar[d]^{b_f}\\
L(B)\coprod \LL(u),\partial_{\alpha} \  \ar[r]_{h_Y}&\ Der  (\Lambda V),\partial_Y
}$$
with

$\bullet$ $\partial_X\mid_{ Der  (\Lambda V, \Lambda V\otimes \Lambda W)}=\partial_Y+\tau$ and  $\partial_X
\mid_{ Der  (\Lambda W, \Lambda V\otimes \Lambda W)}=\tau$ for some 
$\tau :  Der_*(\Lambda V\otimes \Lambda W)\to Der_{*-1}  (\Lambda W, \Lambda V\otimes \Lambda W)$
and 

$\bullet$ $h_X=h_X'+h_X''$ where $h_X'(b)\in Der  (\Lambda V)$ and $h_X''(b)\in Der  (\Lambda W, \Lambda V\otimes \Lambda W)$
for $b\in L(B)$.
 \end{thm}

\begin{proof}
Since $b_f\circ h_X= h_Y\circ i$ and $h_Y$ is a DGL-map, 
 $$h'_X\partial_{\alpha}(u)= h_Y\partial_{\alpha}(u)= \partial_{Y}h_Y(u)\ \ \ \ \ \ (1)$$
in $Der  ( \Lambda V)$. 
Notice that the obstruction element $\partial_X (h_Y(u))-h_X(\partial_{\alpha}(u))$
is a $\partial_X$-cycle in   $Der  ( \Lambda V\otimes \Lambda W)$.
Therefore $\tau (h_Y(u))-h''_X(\partial_{\alpha}(u))$ is a $\partial_X$-cycle
in $Der  ( \Lambda W,  \Lambda V\otimes \Lambda W)$ from (1).\\
(if) Suppose that  ${\mathcal O}_{\alpha}(h_X, h_Y)=0$. 
Then there is an element $q\in Der  ( \Lambda W, \Lambda V\otimes \Lambda W)$ such that
$$\partial_X(q)= \tau (h_Y(u))-h''_X(\partial_{\alpha}(u)).\ \ \ \ \ \ \ \ (2)$$
Let $$h\mid_{L(B)} :=h_X \ \mbox{ and } \ \ h(u):=h_Y(u)-q$$
Then $h$ is a DGL-map since $$\partial_{X}(h(u))=\partial_X(h_Y(u))-\partial_X(q)$$
$$=(h'_X\partial_{\alpha}(u)+\tau h_Y(u))- (\tau h_Y(u)-h''_X\partial_{\alpha}(u))$$
$$=h'_X\partial_{\alpha}(u)+h''_X\partial_{\alpha}(u)=h_X(\partial_{\alpha}(u))=h(\partial_{\alpha}(u))$$
from (1) and (2).
Furthermore  $$\xymatrix{L(B)\ar[r]^{h_X\ \ \ \ \ \  }\ar[d]_i& Der  (\Lambda V\otimes \Lambda W),\partial_X\ar[d]^{b_f}\\
L(B)\coprod \LL(u),\partial_{\alpha} \  \ar[r]_{h_Y}\ar@{.>}[ur]^{{h}}&\ Der  (\Lambda V),\partial_Y
}\ \ \ \ \ \ (*)$$
is  commutative since $b_f(q)=0$. Thus the (if)-part holds from  the special realization of $(*)$. 
\\
(only if) 
Suppose that  there exists a map $h$ such that $(*)$ is  commutative.
Since $h$ is a DGL-map,  $$\partial_X(h(u)) =h(\partial_{\alpha}(u))\ \ \ \ \ \ (3)$$ in $Der  ( \Lambda V\otimes \Lambda W)$
and   $$h_X''\partial_{\alpha}(u)= \tau h (u)\ \ \ \ \  \ \ \  \ (4)$$
from (1) and (3).
Furthermore 
$$\tau h (u) \sim \tau h_Y(u)\ \ \ \ \ \  \ \ \ \ (5)$$
in $Der  (  \Lambda W, \Lambda V\otimes \Lambda W)$.
Here $\sim$ means ``homologous''.
Indeed,
(5) follows
since
$$h(u)=h_Y(u)+x$$
for some element $x\in Der  ( \Lambda W, \Lambda V\otimes \Lambda W)$  from $b_f\circ h= h_Y$
and then since
$$ \tau h(u)=\tau (h_Y(u)+x)=\tau h_Y(u)+\partial_X( x).$$
Thus we obtain that  ${\mathcal O}_{\alpha}(h_X, h_Y)=[\tau (h_Y(u))-h''_X(\partial_{\alpha}(u)) ]=0$
from (4) and (5).
\end{proof}





From Theorem \ref{fiber}, we have

\begin{cor}
If $\pi_{\geq N-1}(Baut_1F_f)_{\Q}=0$ for the homotopy fiber $F_f$ of $f$,
there exists a lift $h$ for  the pair $(h_X,h_Y)$ of above.
\end{cor}

\begin{exmp}
Let $B=S^2=\C P^1$.
Let $S^3\times S^5\to X\overset{f}{\to} Y=S^3$ be the fibration given by the model
$$(\Lambda (v),0)\to (\Lambda (v,w_1,w_2),D)\to (\Lambda (w_1,w_2),0) $$
with $|v|=|w_1|=3$, $|w_2|=5$, $Dw_1=0$ and $Dw_2=vw_1$.
Let $L(\C P^2)=L (B\cup_{\alpha} e^4)=(\LL (u_1,u_2),\partial)$
with  $|u_1|=1$, $|u_2|=3$, $\partial u_1=0$ and $\partial u_2=[u_1,u_1]$ \cite{T}.
Let 
$$\xymatrix{S^2\ar[r]^{h_X\ \ \ }\ar[d]_i&  (Baut_1X)_0\ar[d]^{a_f}\\
S^2\cup_{\alpha} e^4 \  \ar[r]_{h_Y}&\ (Baut_1Y)_0
}$$
be  a commutative diagram given by the DGL-model
$$\xymatrix{(\LL (u_1),0)\ar[r]^{h_X\ \ \ \ \ \  }\ar[d]_i& (Der  (\Lambda (v,w_1,w_2)),\partial_X)\ar[d]^{b_f}\\
(\LL (u_1,u_2),\partial )\  \ar[r]_{h_Y}&\ (Der  (\Lambda (v)),0)
}$$
by
$h_X(u_1)=h_Y(u_1)=0$ and $h_Y(u_2)=(v,1)$.
Then ${\mathcal O}_{\alpha}(h_X, h_Y)\neq 0$ in $H_2(Der  (\Lambda (w_1,w_2), \Lambda (v,w_1,w_2))$
since
$$\tau h_Y(u_2)=\partial_X(v,1)=(w_2,w_1)\not\sim 0=h_X''([u_1,u_1])=h_X''(\partial_{\alpha}(u_2)).$$
Thus $h_Y:\C P^2\to  (Baut_1Y)_0$ cannot lift to  $h:\C P^2\to  (Baut_1X)_0$. 
Note that $h_Y$ is extended to $\C P^{\infty}\to  (Baut_1Y)_0$.
Since $BS^1=\C P^{\infty}$, we obtain that  any  free $S^1$-action on $Y$ cannot lift to $X$
 (See \S 5).
\end{exmp}

\section{ An application to lifting actions }

Let $BG$ and $EG$ be  the classifying space and the universal space of a compact connected Lie group  $G$ of $\rank \ G=r$, respectively.
If  $G$ acts on a space $Y$
by $\mu :G\times Y\to Y$, there is the Borel fibration
$$
Y \overset{i}{\to} EG \times_{G}^{\mu} Y \to BG,
$$
where the Borel space 
$ EG \times_{G}^{\mu} Y $ (or simply $ EG \times_{G} Y $)  is the orbit  space of the diagonal  action
$g(e,y)=(eg^{-1},gy)$  
on the product $ EG \times  Y $.
It is rationally given by the  KS extension (model)
$$
(\Q[t_1,\dots,t_r],0)
 \to (\Q[t_1,\dots,t_r] \otimes \Lambda {V},D_{\mu})
 \to (\Lambda {V},d)=M(Y)\ \ \ \ (*)$$
where     $|{t_i}|$ are even for $i=1,\dots,r$, $D_{\mu}(t_i)=0$ and
$D_{\mu}(v) \equiv d(v)$ modulo the ideal $(t_1,\dots,t_r)$ for $v\in V$.

Recall the lifting theorem of D. H. Gottlieb:
\begin{thm}\cite[Theorem 1]{Gol}\label{Gott}
Let a topological group  $G$ acts on a space $Y$.
A fibration $X\overset{f}{\to} Y$ is fibre homotopy equivalent to a $G$-fibration if and only if  it is fibre homotopy equivalent to the pull-back of a fibration over $EG\times_GY$ induced by the inclusion $i: Y\to EG\times_GY$.
\end{thm}

\noindent
{\it Proof of Theorem \ref{lift}.}
Let $h_Y: BG\to (Baut_1Y)_0$ be the rationalization of the classifying map of the Borel fibration 
$Y\overset{i}{\to} EG \times_{G}^{\mu} Y \to BG$ of the action $\mu : G\times Y\to Y$.
Let $B^n$ be the n-skelton of a CW complex $B$.
From Theorem  \ref{obstr} there is a lift $h_X^{\alpha}$ such that 
$${\small \xymatrix{(BG)^n\ar@{.>}[r]^{h_X^n }\ar[d]&  (Baut_1X)_0\ar[d]^{a_f}\\
(BG)^n\cup_{\alpha}  \vee_ie_i^{n+1} \  \ar[r]_{h_Y^{\alpha}}\ar@{.>}[ur]^{h_X^{\alpha}}&\ (Baut_1Y)_0
}}$$ is commutative for all $n$ and attachings $\alpha$  since 
${\mathcal O}_{\alpha}(h_X^n, h_Y^{\alpha})=0$.
Indeed, $\pi_{odd}(Baut_1f)_{\Q}=0$ and 
 $L(BG)$ is oddly graded since $H^*(BG;\Q )$ is evenly graded.
Thus we have the commutative diagram:
$${\small \xymatrix{BG\ar@{.>}[r]^{ {h_X} \ \ }\ar@{=}[d]&  (Baut_1X)_0\ar[d]^{a_f}\\
BG \  \ar[r]_{h_Y\ \ }&\ (Baut_1Y)_0
}}$$
From Theorem \ref{xy}
there is a commutative digram:
$${\small  
\xymatrix{  {E}\ar@{.>}[d]_{g}\ar[r]& BG_0\ar@{=}[d]\\
 (EG\times_GY)_0\ar[r] &BG_0.
\\
}}$$
for some space $E$.
Let  $g':E'\to EG\times_GY$ be the pull-back of $g$ by the rationalization $l_0$ and $f':X'\to Y$ be the  pull-back of $g'$ by $ i$: 
$${\small  
\xymatrix{ X'\ar@{.>}[d]_{f'}\ar[r]^{}&
 E'\ar@{.>}[d]_{{g' }}\ar[r]^{}& E\ar[d]^{g}\ar[r]& BG_0\ar@{=}[d]\\
Y\ar[r]^{i\ \ }& 
{ EG\times_GY}\ar[r]^{l_0} & (EG\times_GY)_0\ar[r] &BG_0.
\\
}}$$
Notice that the model is given by the DGA-commutative digram:
$${\small  
\xymatrix{R\ar@{=}[d]\ar[r]^{}&
R\otimes \Lambda V \ar[d]_{{M(g) }}\ar[r]^{=}& R\otimes \Lambda V\ar[d]_{M(g')}\ar[r]& \Lambda V\ar[d]_{M(f')}\\
R\ar[r]^{}& 
{ R\otimes \Lambda V\otimes \Lambda W}\ar[r]^{=} &R\otimes \Lambda V\otimes \Lambda W\ar[r] &\Lambda V\otimes \Lambda W
\\
}}$$
for $R:=H^*(BG;\Q )=\Q [t_1,\cdots ,t_r]$.
Notice that the third square is given  by the push-out \cite[Proposition 15.8]{FHT}.  
Thus,  from Theorem \ref{Gott}, we obtain  the commutative diagram
$${\small  
\xymatrix{ G\times X'\ar[d]_{f'\times id_G}\ar@{.>}[r]^{\exists}&
X'\ar[d]_{{f' }}\ar[r]^{l_0}& X'_0\ar[d]_{f'_0}\ar[r]^{\simeq}& X_0\ar[ld]^{f_0}\\
G\times Y\ar[r]^{\mu }& 
{ Y}\ar[r]^{l_0} & Y_0 &
\\
}}$$
since $M(X')\cong  \Lambda V\otimes \Lambda W=M(X)$.
\hfill\qed\\ 

If the  $r$-torus $T^r$ acts on a space $Y$, 
$|t_1|=\cdots =|t_r|=2$  in $(*)$.

\begin{prop}\cite[Proposition 4.2]{H}\label{ha}
Suppose that $Y$ is a simply connected CW-complex  with 
$\dim H^*(Y;\Q)<\infty$.
Put $M(Y)=(\Lambda V,d)$.
Then  $r_0(Y) \ge r$ if and only if there is a KS extension $(*)$
 satisfying $\dim H^*(\Q[t_1,\dots,t_r] \otimes \wedge{V},D)<\infty$.
 Moreover,
 if  $r_0(Y) \ge r$,
 then $T^r$ acts freely on a finite complex $Y'$
 that has  the same rational homotopy type as $Y$
 and $M(ET^r\times_{T^r}Y')\cong 
(\Q[t_1,\dots,t_r] \otimes \wedge{V},D)$.
\end{prop}

\noindent
{\it Proof of Corollary \ref{coro}.}
Let $r_0(Y)=r$. 
Notice that 
$L(BT^r)$ is oddly generated since
$H^*(BT^r;\Q )=\Q [t_1,\cdots ,t_r]$. 
Since  $\pi_{odd}(Baut_1f)_{\Q}=0$, there exists   a lift $(BT^r)_0\to (Baut_1X)_0$
from  Theorem \ref{obstr} (Corollary \ref{X-Y}).
Then we have the homotopy commutative diagram:
$$\xymatrix{(F_f)_0\ar[r]\ar[d]^{\simeq}&X_0\ar[d]\ar[r]^{f_0}&\ar[d]Y_0\ar[r]^{\simeq } &Y'_0\ar[d]\\
F_g\ar[d]\ar[r]&\tilde{E}\ar@{.>}[r]^{\exists{g}\ \ \ \ }\ar[d]&E\ar[d]\ar[r]^{\simeq\ \ \ \ \ }&(ET^r\times_{T^r}Y')_0\ar[d]\\
\bullet \ar[r]&(BT^r)_0\ar@{=}[r]&(BT^r)_0\ar@{=}[r]&(BT^r)_0
}$$
 from Theorem \ref{xy}.
Here $\bullet$ is the one point space.
We have  $\dim H^*(\tilde{E};\Q )<\infty$ since $\dim H^*(F_g;\Q )<\infty$ and $\dim H^*(E;\Q )=\dim H^*(ET^r\times_{T^r}Y;\Q )<\infty$
for the fibration $F_g\to \tilde{E}\to E$.
Thus there is a free $T^r$-action on $X'$
with $X'_0\simeq X_0$ and $\tilde{E}\simeq (ET^r\times_{T^r}X')_0$ from Proposition \ref{ha}.
Thus we have $r_0(X)\geq r$.
\hfill\qed\\

\begin{exmp} 
Let $S^5\to X\overset{f}{\to}  Y$ be a rationally non-trivial fibration
given by the model  $$(\Lambda V,d_Y)=(\Lambda (v_1,v_2,v_3,v_4,v_5),d_Y)\to (\Lambda (v_1,v_2,v_3,v_4,v_5,w),D)  $$
with $|v_1|=|v_2|=2$, $|v_3|=|v_4|=|v_5|=|w|=5$, 
$d_Y(v_1)=d_Y(v_2)=0$, $d_Y(v_3)=v_1^3$, $d_Y(v_4)=v_1^2v_2$, $d_Y(v_5)=v_2^3$
and  
$D(w)=v_1v_2^2$.
Then $$\pi_{odd} (Baut_1f)_{\Q}\cong H_{even}(Der_{\Lambda V}(\Lambda V\otimes \Lambda (w)))=0$$
since there is no element of odd-degree $<5$ in $\Lambda V$.
Therefore  $r_0(Y)\leq r_0(X)$.
Indeed, we can  directly check that $r_0(Y)=1$ and $r_0(X)=2$.

On the other hand, 
let $S^5\to X\overset{f}{\to}  Y=S^3\times S^3$ be a rationally non-trivial fibration.
Then the model is given by
$$(\Lambda (v_1,v_2),0)\to (\Lambda (v_1,v_2,w),D)  $$
with $|v_1|=|v_2|=3$, $|w|=5$, 
and  
$D(w)=v_1v_2$.
Then $$\pi_{3} (Baut_1f)_{\Q}\cong H_{2}(Der_{\Lambda V}(\Lambda V\otimes \Lambda (w)))=\Q \{ (w,v_1)\} \oplus \Q  \{ (w,v_2)\} \neq 0$$
and   $r_0(Y)=2>1= r_0(X)$.



\end{exmp}

Finally recall \cite{Y2}.
For a map $f:X\to Y$, we say  that  the {\it rational toral rank of $f$},
denoted as  $r_0(f)$,    
 is $r$ when it is  the largest  integer  such that there is  a map $F$ between 
an $X_0$-fibration and a $Y_0$-fibration  over $(BT^r)_{0}$  as
$$ {\small \xymatrix{X_{0}\ar[d]_{i_1}\ar[r]^{f_{0}}& Y_{0}\ar[d]^{i_2}\\
E_1\ar[d]_{p_1}\ar@{.>}[r]^{F} &E_2\ar[d]^{p_2}\\
(BT^r)_{0}\ar@{=}[r]&(BT^r)_{0}
}}$$
with $\dim H^*(E_i;\Q)<\infty$ for $i=1,2$.
From the definition, 
$r_0( f)\leq \min \{  r_0(X),r_0(Y) \}$
for any map $f:X\to Y$.
From the proof of Corollary \ref{coro},
we obtain

\begin{cor}\label{} Let $f:X\to Y$  
be a $\pi_{\Q}$-separable map 
with  $Y$ and $F_f$  finite.
If  $\pi_{odd}(Baut_1f)_{\Q}=0$, we have
$r_0(f)= r_0(Y)$. 
\end{cor}

\end{document}